\documentclass[11pt, leqno]{article}
\usepackage{amsmath,amsthm,amsfonts,amssymb,tabularx}
\usepackage[latin1]{inputenc}
\usepackage{graphicx}
\usepackage{color}
\usepackage{nameref}
\usepackage{pgf,tikz}
\usepackage{pdfsync}
\usetikzlibrary[patterns]
\usetikzlibrary{arrows}
\usepackage{hyperref}
\usepackage{pdfpages}
\hypersetup{urlcolor=blue, colorlinks=true}
\usepackage[inner=3cm,outer=4.5cm,bottom=4cm,top=4cm]{geometry}

\usepackage{mathptmx}       
\usepackage{helvet}         
\usepackage{courier}        
\usepackage{type1cm}        
\usepackage{makeidx}         
\usepackage{graphicx}        
\usepackage{multicol}        
\usepackage[bottom]{footmisc}

\usepackage{amsmath}
\usepackage{amssymb}
\usepackage{amscd}
\usepackage{indentfirst}

\def\1{{\bf 1}}

\newtheorem{theorem}{Theorem}[section]
\newtheorem{proposition}[theorem]{Proposition}

\newtheorem{remark}{Remark}
\newtheorem{definition}[theorem]{Definition}
\newtheorem{lemma}[theorem]{Lemma}

\newcommand\R{\mathbb{R}}
\newcommand\RR{\mathbb{R}^N}

\newcommand\lapal{(-\Delta)^{\alpha}}

\newcommand\intr{\int_{\mathbb{R}^N}}

\newcommand\m{m}  

\newcommand\mb{\overline{m}}

\newcommand{\RN}{\mathbb{R}^N}
\newcommand\s{s}  

\newcommand\SqLop{ \left( (-\Delta)^{-1}\mathcal{L}^{1-s}_\epsilon \right)^{\frac{1}{2}} }

\newcommand\x{\overline{x}}

\newcommand\ula{u_\lambda}
\newcommand\betau{\overline{\beta}}
\newcommand\alphau{\overline{\alpha}}

\newcommand{\mf}{q}   
\newcommand{\uf}{w}   
\newcommand{\sfp}{\sigma}   

\makeindex             
\begin{document}

\title{ Porous medium equation with nonlocal pressure}

\author{\large \bf \centerline{Diana Stan, F\'elix del Teso, and Juan Luis V\'azquez}}

\maketitle
\begin{center}
\emph{\footnotesize Dedicated to Profs. Haim Brezis and Louis Nirenberg with deep admiration}
\end{center}

\

\abstract{We provide a rather complete description of the results obtained so far on the nonlinear diffusion equation $u_t=\nabla\cdot (u^{m-1}\nabla (-\Delta)^{-s}u)$, which describes a flow through a porous medium driven by a nonlocal pressure. We consider  constant parameters  $m>1$ and $0<s<1$,  we assume that  the solutions are non-negative,  and the problem is posed in the whole space.
We present a theory of existence of solutions, results on uniqueness,  and relation to other models.  As new results of this paper, we prove the existence of self-similar solutions in the range when $N=1$ and $m>2$, and the asymptotic behavior of solutions when $N=1$. The cases $m = 1$ and
$m = 2$ were rather well known.
}

\vspace{1.3cm}

\vspace{0.3cm}


\small
\vspace{1.5cm}

\noindent \textbf{Addresses:}\\
Diana Stan, {\tt dstan@bcamath.org}, Basque Center for Applied Mathematics, Basque-Country, Spain. \\
 F\'{e}lix del Teso, {\tt felix.delteso@ntnu.no}, Norwegian University of Science and Technology, Trondheim, Norway. \\
  Juan Luis V{\'a}zquez, {\tt juanluis.vazquez@uam.es}, Departamento de Matem\'{a}ticas, Universidad Aut\'{o}noma de Madrid, Madrid, Spain.

\

\newpage
\tableofcontents
\newpage

\section{Introduction}

The study of evolution equations associated to dissipative operators was intensely pursued in the second half of the 20th century. It took the form of the abstract equation
\begin{equation*}
\frac{du}{dt}= A(u)+f\,,
\end{equation*}
where $u$ is a time function with values in a Hilbert or Banach space $X$ and $A$ is typically an unbounded linear or nonlinear operator with strict dissipative conditions, the simplest case being the Laplacian operator acting on $L^2(\Omega)$. Famous theorems were proved that widely extended what was known for the heat equation and then for parabolic equations. The first results dealt with the case where the  underlying functional space $X$ was a Hilbert space, \cite{BrBk73, BrBkFA, Yosida}, and then the theory applied to Banach spaces, \cite{Nir63, CL71, Ev78}. The aim of the general theory was to construct the corresponding semigroups (for $f=0$) or flows (for general $f$) with detailed properties \cite{Gold85, LadBk, Lunardi, PazyBk}.

\medskip

\noindent {\bf I.} It was soon realized that  a general theory was bound to be too rich in details and difficulties, and this led to concentrating the attention on particular equations with a relevant physical interest and significant novel properties, \cite{BrBk71, LiMag}. One of these equations was the porous medium equation, PME for short,
\[
\partial_t u=\Delta(u^m)=\nabla( mu^{m-1}\nabla u), \quad m>1,
\]
which had received much attention from the Russian school in the 1950's and 60's, \cite{B52, OKC}, and was taken up in the abstract general setting in the 70-80's, \cite{BeTh72, BBC75}. This equation is a relevant model for heat propagation with finite propagation speed. As a consequence of such property, interesting geometry occurs in the form free boundaries (interfaces), an issue that originated a great amount of mathematical analysis \cite{CaFr79, CaFr80, CVW87, Koch99, KKochV}.
Work done for several decades led to considerable success and a rather complete theory was formulated, cf. \cite{VazBkPME}. The PME generates a nonlinear contraction semigroup in $X=L^1(\Omega)$, $\Omega\subset \RR$ with definite regularity properties.

 This led to extensions like the fast diffusion equation where $m<1$, the $p$-Laplacian equation, $\partial_t u=\Delta_p(u)=\nabla(|\nabla u|^{p-2}\nabla u)$, $1<p<\infty$,  \cite{diBen93, VazBk06},  and several other equations in the areas of Nonlinear Diffusion and Reaction-Diffusion. Much attention was also given to the Stefan Problem, very important for its mathematics and its applications, \cite{CaffEvans, CaffSa05, diBen93, KamStefan, Meir}.  For a general survey paper see \cite{VazCIME}.

\medskip

\noindent {\bf II.} A decade ago  there arose the interest in combining the PME nonlinear mechanism with
 nonlocal operators so as to take into account anomalous diffusion effects, the main examples of such operators being the so-called fractional Laplacian operators. This trend has been a very active area of research since then. The main problem to be treated was the Porous Medium Equation with fractional pressure formulated by Caffarelli and V\'azquez as
 \begin{equation}\label{CV}
\partial_t u=\nabla(u\nabla (-\Delta)^{-s}u),
\end{equation}
for $0<s<1$,  having in mind models in statistical mechanics, see \cite{GL1, GL2} that deal with  the macroscopic evolution of interacting gas systems. In \cite{CVf1} existence of weak solutions is proved for initial data $u_0 \in L^1\cap L^\infty$  via an approximation method that requires a suitable decay of the data. On the other hand, the 1D model was investigated by Biler et al. in \cite{BilerKarchMonneau} having an application to dislocation theory, \cite{Head}.

Uniqueness is  a key issue for this model. It holds for suitable solutions in 1D as  proved in \cite{BilerKarchMonneau}. However, uniqueness of weak or better solutions is not known so far in several dimensions, except  locally in time for good data  (see \cite{ZXC} for Besov spaces).  The lack of comparison principle is also an important issue to deal with.

A quite important, and rather surprising, feature of the equation is the property of finite speed of propagation. This property was not evident since fractional operators are known to imply very fast propagation. and could overrun the PME nonlinearity, but the nonlinearity wins in this case. The property was proved \cite{CVf1} by comparison with special type of barrier functions, called by the authors true supersolutions. H\"{o}lder regularity of solutions is proved in \cite{CSV,CVf3}. The large time asymptotic behavior of weak solutions of \eqref{CV} is given by a unique fundamental solution constructed via an obstacle type problem (the proof is  given in \cite{CVf2}). A suitable entropy function is constructed to prove the uniqueness of the self-similar solution and the asymptotic behavior.
 Refined asymptotics was done in \cite{CarrilloHuangVazquez}. Gradient flow methods are also an alternative to prove existence of solutions, see \cite{Lisini}. However, because of the lack of a uniqueness theory, the constructed weak solutions of \cite{CVf1} and gradient flow solutions (\cite{Lisini}) might be different.

\medskip

\noindent {\bf III.} Extending model \eqref{CV} to general exponents, as in the standard PME, is natural and motivates the model we study here:
\begin{equation}\tag{M1}
 \partial_t u = \nabla \cdot (u^{m-1} \nabla (-\Delta)^{-s}u) \qquad x\in\R^N, \ t>0.
\end{equation}
The extension also agrees with the generality of the models proposed in \cite{GL1, GLM2000}. We are specially interested in better understanding well-posedness and velocity of propagation.
It turns out that this equation has quite interesting properties, some of them are inherited from \eqref{CV} (finite speed of propagation for $m\ge 2$) and therefore some techniques of the proofs can be successfully adapted, but many other different properties were discovered, like infinite speed of propagation for $m \in (1,2)$. In what follows we describe this last model with its main properties that have been obtained so far: existence of solutions in the general setting of finite measure data (therefore this extends the result of \cite{CVf1} for \eqref{CV}), the transition finite-to-infinite speed of propagation, uniqueness in dimension one, and so on. Moreover we prove new results like existence of selfsimilar solutions for $m\geq2$ and $N=1$, the asymptotic behavior for $m>1$ and $N=1$ and partial results on these topics in higher dimensions.

\medskip

\noindent {\bf IV. }  Equations with two nonlinearities of the form
\[
\partial_t u = \nabla \cdot (u^{a-1} \nabla (-\Delta)^{-s}u^{b-1}) \quad \tag{GM}\label{GM}
\]
have also been considered in the literature. We refer to \cite{StTeVa15} for construction of self-similar solutions when $a>1$, $b>1$ and transformation formulas between self-similar solutions of \eqref{GM}, \eqref{eq:maineq} and \eqref{FPME}.

Dolbeault and Zhang (\cite{DolbeaultZhang}) proved that for $a-1=\frac{1}{2}$, $b<\frac{3}{2}$, self-similar solutions are not optimal for the Gagliardo-Nirenberg-Sobolev inequalities, in strong contrast with usual standard fast
diffusion equations based on non-fractional operators (see \cite{DelPinoDolb}). Their approach is mainly based on entropy methods.

Important related work is due to Biler, Imbert and Karch. In \cite{BilerImbertKarch1, BilerImbertKarch} they considered the model
  \begin{equation}\label{eq:biler}
    \partial_t u = \nabla \cdot (|u| \nabla (-\Delta)^{-s}|u|^{b-2}u),
  \end{equation}
  corresponding to $a=2$, $b>1$ in \eqref{GM}. They prove existence of a changing-sign weak solution and its main properties. Moreover, they find explicit self-similar solutions with compact support. In a later work \cite{Imb16}  finite speed of propagation is established for general positive solutions.

\section{Presentation of the model}
We consider the initial value problem
\[
   \left\{ \begin{array}{ll}
  \partial_t u = \nabla \cdot (u^{m-1} \nabla (-\Delta)^{-s}u)    &\text{for } x \in \RN, \, t>0,\\[2mm]
  u(0,x)  =u_0(x) &\text{for } x \in \RN,
    \end{array}  \right. \tag{M1}\label{eq:maineq}
\]
for $u=u(x,t)\geq0$, exponents $m > 1$, $0<s<1$, and space dimension $N\geq 1$.

This model has been studied in the series of papers \cite{StTeVa14,StTeVa15,StTeVa16,StTeVa17} by the present authors. Many properties of the solutions were proved and there is of course  work to be done. In this paper  we report on the results obtained so far and also we make a step further in the theory of this model by describing the asymptotic behavior in dimension $N=1$.

First, we introduce the notion of weak solution in the very general context of  measures as initial  data.
It applies for all $m\in (1,+\infty)$ and $s\in (0,1)$. We denote by $\mathcal{M}^{+} (\RN)$ the set of nonnegative Radon measures.

\begin{definition}\label{def1}
We say that $u\ge0$  is a weak solution of Problem \eqref{eq:maineq} with initial data $\mu\in  \mathcal{M}^{+} (\RN)$ \ if : \ $u \in L^1_{\textup{loc}}(\RN \times (0,T))$,  $\nabla (-\Delta)^{-s}u \in L^1_{\textup{loc}}(\RN \times (0,T))$, $u^{m-1} \nabla (-\Delta)^{-s} u\in L^1_{\textup{loc}}(\RN \times (0,T))$, and
\begin{equation*}
\int_0^T\int_{\RN} u \phi_t\,dxdt-\int_0^T\int_{\RN}  u^{m-1} \nabla (-\Delta)^{-s} u \cdot\nabla \phi \,dxdt+  \int_{\RN} \phi(x,0) d \mu(x)=0,
\end{equation*}
for all test functions  $\phi \in C^1_c(\RN \times [0,T))$.
\end{definition}
\begin{remark}
Note that, if $\mu = u_0\in L^1_{\textup{loc}}(\R^N)$, then $d \mu(x)=u_0(x)dx$  and
\[\int_{\RN} \phi(x,0) d \mu(x)=\int_{\RN} u_0(x)\phi (x,0) dx,\]
thus, the initial   datum is taken in the usual sense (as initial trace).
\end{remark}

In what follows we will present the main results on Problem \eqref{eq:maineq}. In Section \ref{Section:RelatedModels} three different diffusion models are introduced and we show their relation to Problem \eqref{eq:maineq} and the consequences of this transformation on the qualitative properties of solutions to Problem \eqref{eq:maineq}. Section \ref{sec:inteq} is devoted to the integrated version of \eqref{eq:maineq}. In Section \ref{Section:Main} we state general results on existence of solutions, velocity of propagation. Uniqueness and asymptotic behavior are also stated, but only in dimension 1. Finally the proofs are given in Section \ref{Section:Proofs}.

\section{Related  models and transformations}\label{Section:RelatedModels}

The study of some properties of  Problem \eqref{eq:maineq} is made more difficult because  the comparison principle does not hold and no proof of uniqueness of weak solutions in dimension higher than one is known. This motivates us to search for a connection with other fractional diffusion models that allow a more friendly approach. We present here some models of fractional diffusion  connected to Problem \eqref{eq:maineq} via useful transformations at the self-similar level.

\subsection{The Fractional Porous Medium Equation}

An alternative fractional version of the standard or local Porous Medium Equation, $\uf_t=\Delta \uf^{\mf}$, is given by the following equation
\[
\uf_t+(-\Delta)^{\sfp} \uf^{\mf}=0, \quad \tag{FPME}\label{FPME}
\]
called the \emph{Fractional Porous Equation}  in the literature (we call the exponent $q$ instead of the usual $m$ for convenience in later comparisons). For ${\mf}=1$ and $0<{\sfp}<1$ this leads to the linear Fractional Heat Equation, for which we refer to the survey \cite{BSV17} and also \cite{Vaz2017}. Note that \eqref{FPME} corresponds to the general model \eqref{GM} when $a=1$, $b=q+1$ and $s=1-\sigma$.

In recent years the theory for this model has been widely developed:  the existence, uniqueness and continuous dependence of solutions of the Cauchy problem \eqref{FPME} for all ${\mf}>0$ and $0<{\sfp}<1$ have been proved by De Pablo, Quir\'os, Rodr\'iguez and V\'azquez in \cite{DPQuRoVa11,DPQuRoVa12}.

 The \eqref{FPME} model inherits some of the properties of the classical PME. Using the Caffarelli-Silvestre
extension method and the B\'enilan-Brezis-Crandall functional semigroup approach, a
weak energy solution is constructed, and $u \in C([0, \infty) : L^1(\RN))$. Moreover, the set of
solutions forms a semigroup of ordered contractions in $L^1(\RN)$.

An important property of \eqref{FPME}, which does not hold in its non-fractional version,  is the\emph{ infinite speed of propagation}: assume ${\sfp}\in (0,1)$, ${\mf}>({\mf})_c=(N-2{\sfp})_+/N$. Then for non-negative initial data $u_0 \ge 0$ such that $\int_{\RN} u_0 (x)\,dx<\infty$, there exists a unique solution ${\uf}(x,t)$ of problem \eqref{FPME} satisfying ${\uf}(x,t)>0$ for all $x\in \mathbb{R}^N$, $t>0$. Moreover, there is conservation of mass $\int \uf(x,t)\,dx= \int \uf_0(x)\,dx$  for all $t>0$.
The solutions are $C^\alpha$ continuous as proved in \cite{VaPQR2017}. More general diffusions were considered in \cite{DPQuRoVa14}. We refer to the survey \cite{VazSurvey2014} for a complete description of the \eqref{FPME}.

\noindent\textbf{Asymptotic behavior and self-similarity for FPME.} The large time behaviour of such solutions is described by the self-similar solutions  with finite mass (Barenblatt solutions) constructed in \cite{Vaz14}, which have the form
$$
\uf(x,t)=t^{-N\beta_1}\phi_1(y), \quad y=x\,t^{-\beta_1}\,,
$$
where $\beta_1=1/(N(\mf-1)+2{\sfp})$ and the profile function $\phi_1$ satisfies the following equation
\begin{equation}\label{prof1}
(-\Delta)^{\sfp}  \phi_1^\mf=\beta_1 \nabla \cdot(y \,\phi_1).
\end{equation}
 These solutions are well defined for  $\mf>(\mf)_c$, where $\beta_1$ is positive. The profile $\phi_1(y)$ is a smooth and positive radial function in $\RN$, it is monotone decreasing in $r=|y|$ and has a polynomial decay rate as $|y|\to \infty$ depending on the exponent $\mf$.

\subsection{Self-similar solutions for Problem \eqref{eq:maineq} when $m<2$}

In a previous work \cite{StTeVa15} we have established three main types of self-similar solutions for model \eqref{eq:maineq} depending on the range of the parameter $\m$, but always restricted to the range $m<2$. The first type are functions that are positive for all times, while the second type are functions that extinguish in finite time, both separated by a transition type. We briefly present these solutions here with the purpose of relating model \eqref{eq:maineq} with its alternative \eqref{FPME}. The rigorous computations and the derivation of the formulas are written in \cite{StTeVa15}.

\medskip

\noindent  \textbf{Self-similarity of first type. Solutions that exist for all positive times.}
A self-similar solution $V(x,t)$ of the first type to equation \eqref{eq:maineq} conserving mass is given by
\begin{equation*}
V(x,t)=t^{-\alpha_2}\phi_2(y), \quad y=x\,t^{-\beta_2}
\end{equation*}
with $\alpha_2=N\beta_2$  and $\beta_2=1/(N(\m-1)+2-2\s)$, and with profile function $\phi_2$ satisfying the equation
\begin{equation}\label{prof2}
\nabla \cdot(\phi_2^{\m-1}\,\nabla(-\Delta)^{-\s}\phi_2 )=-\beta_2 \nabla \cdot(y \,\phi_2).
\end{equation}

These solutions  are considered in the range of parameters where $\beta_2>0$, that is, for $\m>(N-2+2\s)/N$.

\medskip

\noindent \textbf{Self-Similarity of second type. Extinction in finite time.}
These are solutions to equation \eqref{eq:maineq} with the self-similar form
\begin{equation*}
V(x,t)=(T-t)^{\alphau_2} \psi_2 \left(y\right), \quad y=x(T-t)^{\betau_2}.
\end{equation*}
where $\alphau_2=N\betau_2,  \, \betau_2=1/(N(1-\m)+2\s-2).$
The profile $\psi_2$ satisfies the equation
\begin{equation}\label{prof5}
\nabla \cdot(\psi_2^{\m-1}\nabla (-\Delta)^{-\s}\psi_2)=\nabla \cdot(y \psi_2).
\end{equation}
Here $\betau_2=-\beta_2$, where $\beta_2$ is the self-similarity exponent of first type. We argue now in the range of parameters where $\betau_2>0$, that is $\m < (N-2+2\s)/N$.

\medskip

\noindent $\bullet$ \textbf{Self-Similarity of third type. Eternal solutions.}
 For $\m\to(N-2+2\s)/N$ there is a class of self-similar solutions to equation \eqref{eq:maineq} conserving mass of the form
\begin{equation*}
V(x,t)=e^{-c \, t}F(y), \quad y=xe^{-c\, t} ,
\end{equation*}
where $c>0$ is a free parameter (exponential self-similarity, which usually plays a transition role) and
$F$ is a solution to the profile equation
\begin{equation}\label{prof6}
\nabla \cdot(F^{\m-1}\nabla(-\Delta)^{-\s}F)=-c\nabla \cdot(y F ).
\end{equation}

\begin{remark} Solutions of this type live backward and forward in time, they are eternal. Notice that in this borderline  case  $\m\to(N-2+2\s)/N$ we have $1/\beta_2=1/\betau_2\to0$, and therefore self-similar solutions of the first and second type do not apply here.
\end{remark}

\noindent $\bullet$ {\bf The transformation.} In \cite{StTeVa15} we found an unexpected relationship that allows to transform the families of mass-conserving self-similar solutions of models \eqref{FPME} and \eqref{eq:maineq} into each other, if suitable parameter ranges are prescribed. Actually, there exists a precise correspondence between the profiles $\phi_1$ and $\phi_2$, $\psi_2$ or $F$, and the parameters $\mf$ and $\m$, as well as  $\sfp$ and $\s$.

\begin{theorem}Let $\mf> \frac{N-2{\sfp}}{N}$, $s\in (0,1)$ and let $\phi_1\ge 0$ be a solution to the profile equation \eqref{prof1}. The following holds:

{\rm (i)} If $\mf \in ( \frac{N}{N+2{\sfp}}, \infty)$ then
\begin{equation*}
\phi_2(x)=\left(\beta_1/\beta_2 \right)^{\frac{\mf}{1-\mf}} (\phi_1(x))^\mf
\end{equation*}
is a solution to the profile equation \eqref{prof2} if we put $\m= \frac{2\mf-1}{\mf}$ and $\s=1-{\sfp}$. This corresponds to $m \in ( \frac{N-2+2\s}{N},2)$, $\s\in (0,1)$.

{\rm (ii)} If $\mf \in  (\frac{N-2{\sfp}}{N}, \frac{N}{N+2{\sfp}})$    then
\begin{equation*}
\psi_2(x)=\left(\beta_1/\beta_2 \right)^{\frac{\mf}{1-\mf}} (\phi_1(x))^\mf
\end{equation*}
is a solution to the profile equation \eqref{prof5} if we put $\m=\frac{2\mf-1}{\mf}$ and $\s=1-{\sfp}$. This corresponds to $m \in ( \frac{N-4+4\s}{N-2+2  \s}, \frac{N-2+2\s}{N})$, $\s\in (0,1)$.

{\rm (iii)} If $\mf= \frac{N}{N+2{\sfp}}$ then
\begin{equation*}
F(x)=\left(\beta_1/c \right)^{\frac{N}{2{\sfp}}} (\phi_1(x))^{\frac{N}{N+2{\sfp}}}
\end{equation*}
is a solution to the profile equation \eqref{prof6} if we put $\m= \frac{N-2+2\s}{N}$ and $\s=1-{\sfp}$.

\end{theorem}

This transformation is an algebraic relation that has important consequences. It provides a proof for the existence of self-similar solutions to Problem \eqref{eq:maineq} and their characterization.
\normalcolor Note that $m<2$. \normalcolor

Self-similar solutions of equation \eqref{FPME} can be also constructed  for smaller values of $\mf$: these are very singular solutions that extinguish in finite time (see V\'azquez \cite{Vaz14}) or blow up in finite time (see V\'azquez and Volzone \cite{VazVol2015}).
 These ones can be transformed in a similar manner to corresponding self-similar solutions of the model \eqref{eq:maineq}. Rigorous proofs with complete computations can  be found in \cite{StTeVa15}. Self-similar solutions do not have an explicit formula, except very particular cases of exponents $m=m(s)$  explicitly computed by Huang in  \cite{Huang2013}.

\subsection{Related models for $m>2$}

In this section we will derive (in a formal way) a transformation formula between self-similar solutions of model \eqref{eq:maineq} and two new nonlocal problems. First, we will show that when $m>2$, self-similar solutions of \eqref{eq:maineq} have a correspondence to self-similar solutions of
 \begin{equation}\label{eq:model2}
 v_t + v^2 (-\Delta)^{1-s}v^{\mb}=0, \quad x\in \RN,\ t>0.
 \end{equation}
with $\mb=1/(m-2)$. The algebraic change $v=1/w$ maps solutions of  problem \eqref{eq:model2} to solutions of
\begin{equation*}
w_t - (-\Delta)^{1-s} w^{-\mb}=0.
\end{equation*}
which corresponds to solutions of \eqref{FPME} for negative exponents $q$ (it is known this equation does not admit integrable solutions, cf. \cite{BSegVaz}, therefore one must search for non-integrable solutions).

Self-similar solutions of Problem \eqref{eq:model2} are of the form
\begin{equation*}
V(x,t)=t^{-a}\psi (y), \quad y=x\,t^{b}
\end{equation*}
with
$$a=bN, \quad b=\frac{1}{N(\mb+1)+2(1-s)},$$
where $\psi$ satisfies the profile equation
\begin{equation}\label{eq:psi}
b(N \psi - y \nabla \psi) = \psi^2(-\Delta)^{1-s}\psi^{\mb}.
\end{equation}

\medskip

\begin{lemma}\label{lem:trans2}
  Let $m>2$ and let $\phi(x,t)$ a smooth solution  to the profile equation \eqref{prof2}. Let $\psi$ and $\mb$ defined by
  \begin{equation*}
  \phi= c \, \psi^{\mb}, \quad \mb=\frac{1}{m-2},
  \end{equation*}
  with $c=\left(\frac{\beta}{b}\right)^{1/(m-1)}$.
  Then  $\psi$ is a solution to the profile equation \eqref{eq:psi}.
\end{lemma}
\begin{proof}
  The proof follows directly from the profile equation \eqref{prof2} . Indeed,
   \begin{align*}
   \phi^{m-1}\,\nabla(-\Delta)^{-s}\phi  & = -\beta y \,\phi\\
  \nabla \cdot \nabla(-\Delta)^{-s}\phi  & = -\beta \nabla\cdot( y \,\phi^{2-m})\\
  (-\Delta)^{1-s}\phi  & = \beta ( N \, \phi^{2-m} +  (2-m)y \phi^{1-m} \nabla \phi)\\
     \phi^{2(m-2)}(-\Delta)^{1-s}\phi & = \beta ( N \, \phi^{m-2} - y  \nabla \phi^{m-2 })\\
\psi^{2}(-\Delta)^{1-s}\psi^{\mb} & =  b (N \, \psi  -  y  \nabla \psi).
 \end{align*}
\end{proof}

\begin{remark}\begin{enumerate}
\item[(i)]  In Lemma \ref{LemmaSelfSim} we will  have prove the existence of a self-similar solution of model \eqref{eq:maineq} when $N=1$ as a limit of the rescaled solutions.  The transformation formula given by Lemma  \ref{lem:trans2} would give the existence of a self-solution of model \eqref{eq:maineq} for $N\geq1$ and $m\geq2$ subject to a rigorous proof of existence of selfsimilar solutions of model  \eqref{eq:model2}.
\item[(ii)] Model \eqref{eq:model2} has not been studied in the literature yet. However, it seems natural to think that it will enjoy some good properties, such as comparison principle for viscosity solutions. It is a possible direction to follow in order to deal with the open problem of the asymptotic behavior of solutions of model \eqref{eq:maineq} for $m>2$ and $N>1$ where no uniqueness is known for \eqref{eq:maineq}.
\end{enumerate}
\end{remark}

\section{The integrated model in one dimension for $m\in(1,\infty)$}\label{sec:inteq}

We consider model \eqref{eq:maineq} in one space dimension
\begin{equation}\label{model1repeat}
\partial_t u= \partial_x \cdot (u^{m-1} \partial_x (-\Delta)^{-s}u),
\end{equation}
for $x\in \mathbb{R}$, $t>0$ and $s\in (0,1)$. We take compactly supported initial data $u_0\ge 0$ such that $u_0 \in L^1_{\text{loc}}(\mathbb{R}).$ 
The "integrated solution" $v$ is defined by
\begin{equation}\label{def.v}
v(x,t)=\int_{-\infty}^x u(y,t)\, dy \ge 0 \quad \text{for }t>0, \ x \in \mathbb{R}.
\end{equation}
Therefore $v_x=u$ and $v(x,t)$ will be a solution (in the viscosity sense) of the equation
\begin{equation}\label{IntegEq}
\partial_t v= -|v_x|^{m-1}\lapal v ,
\end{equation}
with $\alpha=1-s$  and  initial data
\begin{equation}\label{initialv0}
v(x,0)=v_0(x):=\int_{-\infty}^x u_0(x)\,dx  \quad \text{ for all }x \in \mathbb{R}.
\end{equation}
Note that $v(x,t)$ is a non-decreasing function in the space variable $x$. Moreover, since $u(x,t)$ enjoys the property of conservation of mass, then $v(x,t)$ satisfies
$$ \lim_{x\to -\infty}v(x,t)=0, \quad  \lim_{x\to +\infty}v(x,t)=M $$
for all $t\geq0$. We consider viscosity solutions $v(x,t)$ of  \eqref{IntegEq}-\eqref{initialv0} in the sense of Crandall-Lions.

\begin{definition}Let $v$ be a upper semi-continuous function (resp. lower\\ semi-continuous function) in $ \R \times (0,\infty)$.

(i) We say that $v$ is a \textbf{viscosity sub-solution} (resp. \textbf{super-solution}) of equation \eqref{IntegEq} on $\R\times (0,\infty)$ if for any point $(x_0,t_0)$ with $t_0>0$ and any $\tau \in (0,t_0)$ and any test function $\varphi \in C^2( \mathbb{R} \times (0,\infty) ) \cap L^\infty( \mathbb{R}\times (0,\infty) )$ such that $v-\varphi$ attains a global maximum (resp. minimum) at the point $(x_0,t_0)$ on
$
Q_\tau=\mathbb{R} \times (t_0-\tau,t_0]
$
we have that
$$
\partial_t \varphi (x_0,t_0)+ |\varphi_x(x_0,t_0)|^{m-1}(\lapal \varphi(\cdot, t_0)) (x_0) \le 0 \quad (\ge 0).
$$
(ii) We say that $v$ is a \textbf{viscosity sub-solution} (resp. \textbf{super-solution})
 of the initial-value problem \eqref{IntegEq}-\eqref{initialv0} on $\R\times (0,\infty)$ if it satisfies moreover at $t=0$
$$
v(x,0)\le  \limsup_{y\to x,\ t \to 0} v(y,t) \quad ( \text{resp. } v(x,0)\ge  \liminf_{y\to x,\ t \to 0} v(y,t)).
$$
We say that $v\in C(\R \times (0,\infty))$ is a \textbf{viscosity solution} if $v$ is a viscosity sub-solution and a viscosity super-solution on $\R\times (0,\infty)$.
\end{definition}
Since equation \eqref{IntegEq} is invariant under translations, the test function $\varphi$ in the above definition can be taken such that $\varphi$ touches $v$ from above in the sub-solution case, resp. $\varphi$ touches $v$ from below in the super-solution case.

Now we state the correspondence of solutions between the two models, together with the regularity estimates that $u$ inherits from $v$.

\begin{proposition}\label{prop:equivalence}
Let $s\in(0,1)$ and $m\geq1$. Let also $u$ be a weak solution for Problem \eqref{model1repeat} with initial data $u_0\in \mathcal{M}_+(\R)$. Then $v$ defined by \eqref{def.v} is a viscosity solution for Problem \eqref{IntegEq}-\eqref{initialv0} and $v\in C(\R \times (0,T))$. If additionally $u_0\in L^\infty(\R)$, then $v\in C(\R \times [0,T))$.
\end{proposition}
\begin{proof}
We start by proving the regularity of $v$. If $u_0\in L^\infty(\R)$ then $u\in  L^\infty(\R\times(0,T))$.  This implies that $v \in L^\infty([0,T]: \text{Lip}(\R))$ since $v_x =u \in L^\infty([0,T]: L^\infty(\R))$. Moreover $v_t \in {L^2([0,T]:L^2(B))} $ by the second energy estimate for $u$. Then the continuity in time is shown by estimating the decay of the time shift:  $ |v(x_0,t_1)-v(x_0,t_0)| \le K |t_1-t_0|^{1/3}$ when $t_1$ and $t_0$ are close. This is done by combining the $L^2$ estimate on $v_t$ and the already proved spatial regularity.

When $u_0\not \in L^\infty(\R)$, we use the smoothing effect given in Theorem \eqref{ThmExistL1} to show that $u\in L^\infty(\R^N \times (\tau, \infty))$ for any $\tau>0$. Consequently, we proceed as before but avoiding $t=0$.

To show that $v$ is in fact a viscosity solution, we consider the regularized problem
\begin{equation}\label{eq:viscreg}
 (v_\delta)_t= \delta \Delta(v_\delta) + |(v_\delta)_x|^{m-1}(-\Delta)^{1-s}v_\delta.
 \end{equation}
It is clear that $v_\delta=\int_{-\infty}^x u_\delta(y,t)dy$ where $u_\delta$ is the classical solution of
\begin{equation*}
\partial_t u_\delta=\delta  \Delta u_\delta+ \partial_x \cdot (u_\delta^{m-1} \partial_x (-\Delta)^{-s}u_\delta),
\end{equation*}
Not that this is the same problem that in the last step of the prove of existence of $u$. Thus, up to a corresponding subsequence, we can pass to the limit in \eqref{eq:viscreg} and show that $v:=\lim_{\delta \to0}v_\delta$ is a viscosity solution of \eqref{IntegEq} since it is a limit of viscosity solutions.
\end{proof}

The standard comparison principle for viscosity solutions holds true (see Chasseigne and Jakobsen \cite{ChassJakob}). We also refer to \cite{BilerKarchMonneau} for more details regarding properties of this integrated model when $m=2$.
\begin{proposition}[Comparison Principle]Let $m\geq1$ and  $\alpha\in (0,1)$. Let $w$ be a viscosity sub-solution and $W$ be a viscosity super-solution of equation \eqref{IntegEq}.
If $w(x,0) \le W(x,0)$, then $w\le W$ in $\R \times (0,\infty)$.
\end{proposition}

The following uniqueness result is also proved in \cite{ChassJakob}. This result is crucial to obtain uniqueness of weak solution in dimension one for $\eqref{eq:maineq}$.
\begin{proposition}\label{prop:uniqinteg}
Let  $m\geq1$ and  $\alpha\in (0,1)$. Then there exists a unique viscosity solution of Problem \eqref{IntegEq} with piecewise continuous initial data.
\end{proposition}

\section{Main results for model (M1)}\label{Section:Main}

\subsection{Existence of solutions}

The most important contribution to the existence theory  was done in \cite{StTeVa17} where we constructed a weak solution in the sense of Definition \ref{def1} in the general setting of initial data any $\mu \in \mathcal{M}^{+} (\RN)$, the space of nonnegative Radon measures on $\RN$ with \em{finite mass}\em. In particular, this includes the case of merely integrable data $u_0 \in L^1(\RN)$.

\begin{theorem}\label{ThmExistL1}
Let  $m\in (1,+\infty)$, $s\in (0,1)$, $N\ge 1$ and     $\mu \in \mathcal{M}^{+} (\RN)$.  Then there exists a nonnegative weak solution $u$ of Problem \eqref{eq:maineq} and for all $\tau>0$,
$$u \in  L^\infty((\tau,\infty):L^1(\RN))\cap L^\infty (\RN \times (\tau, \infty)) \cap  L^\infty((0,\infty):  \mathcal{M}^{+} (\RN) ).$$
Moreover, $u$ has the following properties:
\begin{enumerate}
\item \textbf{(Conservation of mass)} For all $0 < t < T$ we have
$\displaystyle{
\int_{\RN}u(x,t)dx= \int_{\RN}d\mu(x).
}$
\item \textbf{($L^{\infty}$ estimate) } $0 < \tau<t < T$ we have  $||u(\cdot,t)||_\infty\leq ||u(x,\tau)||_\infty$.

\item \textbf{($L^p$ - energy estimate)} For all $1<p<\infty$ and $0 <\tau< t < T$ we have
\begin{equation}\label{eq:lpenergy}
\begin{split}\intr u^p(x,t)dx + \frac{4p(p-1)}{(m+p-1)^2}\int_\tau^t &\int_{\RN}\Big|(-\Delta)^{\frac{1-s}{2}} \left[u^{\frac{m+p-1}{2}}\right](x,s)\Big|^2dxds \\
&\le \intr u^p(x,\tau)dx .
\end{split}
\end{equation}

\item \textbf{(Second energy estimate)} For all $0 < \tau<t < T$  we have
\begin{equation}\label{eq:secondenergy}
\begin{split}
  \frac{1}{2} \intr \left|(-\Delta)^{-\frac{s}{2}} u(x,t)\right|^2 dx  +
&\int_\tau^t \intr  u^{m-1} \left| \nabla  (-\Delta)^{-s} u(x,s)\right|^2 dx  ds\\
& \le \frac{1}{2} \intr \left|(-\Delta)^{-\frac{s}{2}} u(x,\tau)\right|^2 dx.
\end{split}
\end{equation}

\item \textbf{(Smoothing effect)} For all $t>0$, we have
\begin{equation*}
\| u(\cdot,t)\|_{L^{\infty}(\RN)} \le C_{N,s,m} \, t^{-\gamma} \mu(\R^N)^{\delta}
\end{equation*}
where
$\gamma=\frac{N}{(m-1)N+2(1-s)}>0$ and  $\delta=\frac{2(1-s)}{(m-1)N+2(1-s)}>0$.

\end{enumerate}

\end{theorem}

\begin{remark}
 If $u_0\in L^1(\R^N)\cap L^\infty(\R^N)$, all the properties of Theorem \eqref{ThmExistL1} hold up to $\tau=0$.

\end{remark}


\subsection{Uniqueness in dimension $N=1$}\label{Section:Uniqueness}

Uniqueness of weak solutions is proved in the one-dimensional case.

\begin{theorem}\label{ThmUnique}
Let  $m\in (1,+\infty)$, $s\in (0,1)$, $N= 1$ and  $\mu \in \mathcal{M}^{+} (\RN)$. Then there exists a unique weak solution to Problem \eqref{eq:maineq}.
\end{theorem}
This theorem identifies the constructed weak solution obtained in Theorem \ref{ThmExistL1} as the unique weak solution to Problem \eqref{eq:maineq}. The proof follows as consequence of  Proposition \ref{prop:equivalence} and Proposition \ref{prop:uniqinteg}.


\subsection{Speed of propagation}

A very interesting property is the finite/infinite speed of propagation of the solution of Problem \eqref{eq:maineq}   depending on the nonlinearity parameter $m$, as proved in \cite{StTeVa16,StTeVa17}.

\begin{theorem}\label{ThmFiniteProp}
Let $m\in[2,\infty)$, $s\in(0,1)$ and $N\geq1$. Assume that $u_0\in L^\infty(\R^N)$ has compact support and let $u$ be the constructed weak solution of Problem \eqref{eq:maineq} given in Theorem \ref{ThmExistL1}. Then, $u(\cdot,t)$ has compact support for all $t>0$, i.e., $u$ has \textbf{finite speed of propagation}.
\end{theorem}

\begin{theorem}\label{ThmInfProp}
Let $m\in(1,2)$, $s\in(0,1)$ and $N=1$. Assume $u_0\in L^1(\R^N)\cap L^\infty(\R^N)$ and let $u$ be the weak solution of Problem \eqref{eq:maineq}. Then, for any $t>0$ and $R>0$, the set $\mathcal{P}_{R,t}=\{x: |x|\ge R,\  u(x,t)>0\}$ has positive measure (even if $u_0$ is compactly supported). This is a  weak form of  \textbf{infinite speed of propagation}. Moreover, if $u_0$ is radially symmetric and non-increasing in $|x|$, then  $u(x,t)>0$ for all $x\in \R$ and $t>0$.

\end{theorem}

\begin{remark}
Note that while in Theorem \ref{ThmFiniteProp} we need to assume that $u$ is the constructed weak solution, in Theorem \ref{ThmInfProp} the uniqueness result for dimension $N=1$ given by Theorem \ref{ThmUnique} ensures that this is the only weak solution. This fact also applies to Theorem \ref{ThmFiniteProp} when $N=1$.
\end{remark}


\subsection{Asymptotic behavior}

Once we know the uniqueness result  of Theorem \ref{ThmUnique} and the existence of solutions for finite measure data, we can prove that there exists a unique fundamental solution to Problem \eqref{eq:maineq}  and it  describes the large time asymptotic behavior of a general class of solutions.

\begin{theorem}[Asymptotic Behavior]\label{thm:main}
Let $m\in (1,\infty)$, $s\in(0,1)$ and $N=1$. Assume that $u_0\in L^1(\R)$ such that $\|u_0\|_{L^1(\R)}=M$ and let $u$ be the corresponding weak solution of \eqref{eq:maineq}. Then
\begin{equation*}
t^{\frac{N(1-\frac{1}{p})}{(m-1)N+2-2s}}\|u( \cdot, t)-U_M(\cdot, t )\|_{L^p(\R^N)}\to0 \quad \textup{as} \quad t \to \infty
\end{equation*}
for any $p>1$, where $U_M$ is the unique self-similar solution of \eqref{eq:maineq} with initial data $\mu= M\delta_0$.
\end{theorem}
 Notice that $U_M$ can be transformed into a self-similar solution of \eqref{FPME} (for $m<2$)  or \eqref{eq:model2} (for $m>2$) as explained in Section \ref{Section:RelatedModels}. In the first case $m<2$ this transformation  allows to obtain  the main properties of $U_M$ from the known properties of the Barenblatt solutions of the \eqref{FPME}, which are derived in \cite{Vaz14}. The precise decay for large $|x|$ of $U_M$ is given in \cite{StTeVa15}, Corollary 3.2.

\section{Proofs of the results}\label{Section:Proofs}

\subsection{Sketch of the proof of existence}

The proof of existence for general initial data given by an integrable measure is a long ride, with several nontrivial steps. In this section we will show the strategies of the proof, together with some details of the main ingredients of it.

We will first prove existence for $u_0\in L^1(\R^N)\cap L^\infty(\R^N)$ via a four steps approximation method by regularized versions of \eqref{eq:maineq}. In this first part we will also obtain very useful energy estimates that ensure compactness, together with an $L^1-L^\infty$ smoothing effect. Afterwards, using the smoothing effect we prove existence for initial data $\mu \in \mathcal{M}_+(\R^N)$ approximating $\mu$ by bounded integrable initial data.

\subsubsection{Existence for $u_0\in L^1(\R^N)\cap L^\infty(\R^N)$}

Formally, we consider an equivalent version of \eqref{eq:maineq} given by:
\begin{equation}\label{eq:maineqmod}
u_t= \nabla \cdot(u^{m-1} \nabla (-\Delta)^{-1} (-\Delta)^{1-s}u) \quad \textup{in} \quad \R^N.
\end{equation}
The idea is to consider a regularized version of \eqref{eq:maineqmod} where all the problematic terms are approximated. More precisely, we add the vanishing viscosity term  $\delta \Delta u$ to \eqref{eq:maineqmod} that ensures good properties of regularity for the solution, we eliminate the degeneracy at the zero level sets by putting $u^{m-1}\sim (u+\mu)^{m-1}$ and we eliminate the singular character of the fractional Laplacian  $(-\Delta)^{1-s}$ approximating it by the zero order pseudo-differential operator
\begin{equation}\label{eq:aproxop}
\mathcal{L}^{1-s}_\epsilon (u)(x)=C_{N,1-s}\int_{\RN}\frac{u(x)-u(y)}{\left(|x-y|^2+\epsilon^2\right)^{\frac{N+2-2s}{2}}}dy.
\end{equation}
Additionally, to ensure existence, we will restrict \eqref{eq:maineqmod} to a bounded domain $B_R$. The approximated problem reads
\[\label{ProblemEpsMuDeltaR}\tag{$P_{\epsilon\delta\mu R}$}
\left\{
\begin{array}{ll}
(U_1)_t= \delta \Delta U_1 +\nabla \cdot((U_1+\mu)^{m-1}\nabla (-\Delta)^{-1} \mathcal{L}^{1-s}_\epsilon [U_1])&\text{in }\ B_R \times (0,T),\\
U_1(x,0)=\widehat{u}_0(x) &\text{in } \ B_R,\\
U_1(x,t)=0 &\text{in }  \ B_R^c\times (0,T),
\end{array}
\right.
\]
depending on the parameters $\epsilon,\delta,\mu, R>0$. We also consider $\widehat{u}_0$ to be a smooth approximation of $u_0$. We say that $U_1$ is a weak solution of \eqref{ProblemEpsMuDeltaR} if

\begin{equation*}
\begin{split}
\int_0^T\int_{B_R} U_1(\phi_t-\delta \Delta \phi)dxdt&-\int_0^T\int_{B_R}  (U_1+\mu)^{m-1} \nabla (-\Delta)^{-1} \mathcal{L}^{1-s}_\epsilon [U_1]\cdot\nabla \phi dxdt\\
&+ \int_{B_R} \widehat{u}_0(x) \phi(x,0)dx=0
\end{split}
\end{equation*}
 for smooth test functions $\phi(x,t)$ that vanish on the spatial boundary $\partial B_R$ and for large $t$. Indeed, existence of smooth weak solutions is proved via mild solutions, i.e, $U_1$ is the fixed point of the following map:
$$
\mathcal{T}(v)(x,t)= e^{\delta t \Delta}u_0(x) + \int_0^t \nabla e^{\delta (t-\tau) \Delta} \cdot  G(v)(x,\tau)  d\tau,
$$
where $G(v)= (v+\mu)^{m-1}\nabla (-\Delta)^{-1} \mathcal{L}_\epsilon^s[v]$. The map,
 $$\mathcal{T}: C((0,T):L^1(B_R) \cap L^\infty(B_R)) \to  C((0,T):L^1(B_R) \cap L^\infty(B_R)) $$ is well defined and it is also a contraction,  thus, Banach contraction principle ensures existence of a fixed point. We refer to \cite{BilerImbertKarch} for a very similar proof in a slightly different context.

Once existence and regularity of the approximated problem are obtained, we provide the solution with the following energy-type estimates, that will give compactness that allow to pass to the limit in all the approximation parameters.

\noindent$\bullet$ $L^p$ \textbf{energy estimates for $1\leq p<\infty$.} For all $0<t<T$ we have that:
\begin{equation}\label{eq:Lpaprox}\begin{split}
\int_{B_R} U_1^p(x,t)dx+ &p (p-1) \int_0^t \int_{B_R}  |(\mathcal{L}^{1-s}_\epsilon)^{\frac{1}{2}} [\Psi (U_1)](x,s)|^2 dx ds\\
 &+ \frac{4(p-1)\delta}{p}\int_0^t \int_{B_R} \left|\nabla (U_1^{p/2})(x,s)\right|^2 dxds \leq \int_{B_R} u_0^p(x)dx
\end{split}
\end{equation}
where $\Psi(z)=\int_0^z \zeta^{\frac{p-2}{2}}(\zeta+\mu)^{\frac{m-1}{2}}d\zeta$. We want to mention that a crucial step in the derivation of \eqref{eq:Lpaprox} relies on the generalized version of the Stroock-Varopoulos Inequality:  Given $\psi: \R \to \mathbb{R}$ such that $\psi \in C^1(\R)$, $\psi' \ge 0$ and $\Psi$ such that  $\psi'=(\Psi')^2$ we have
\begin{equation*}
\begin{split}
\int_{\RN}\psi(w)\mathcal{L}^s_\epsilon [w] dx \ge \int_{\RN}\left|(\mathcal{L}^s_\epsilon)^{\frac{1}{2}}[\Psi (w)] \right|^{2} dx.
\end{split}
\end{equation*}

\noindent $\bullet$ \textbf{Second energy estimate.} For all $0<t<T$ we have that:
\begin{align*}
& \frac{1}{2} \int_{B_R}\left|\SqLop [U_1(t)]\right|^2 dx \\
&\quad  + \int_0^t \int_{B_R}(U_1+\mu)^{m-1} \left| \nabla (-\Delta)^{-1} \mathcal{L}^{1-s}_\epsilon [U_1]\right|^2 dx  dt \\
&\quad + \delta \int_0^t\int_{B_R}\left|(\mathcal{L}^{1-s}_\epsilon)^{\frac{1}{2}}[U_1]\right|^2  dx dt  \leq  \frac{1}{2} \int_{B_R}\left|\SqLop [u_0]\right|^2 dx.
\end{align*}

\noindent $\bullet$ \textbf{(Decay of total mass)} For all $0<t<T$ we have
$\displaystyle{
\int_{B_R}U_1(x,t)dx \le\int_{B_R}u_0(x)dx.
}$

\noindent $\bullet$ \textbf{($L^{\infty}$-estimate) } For all $0<t<T$ we have  $||U_1(\cdot,t)||_\infty\leq ||u_0||_\infty$.

 \medskip

By combining these energy estimates we are able to apply some suitable parabolic compactness theorems  to derive convergence of approximated solutions when the parameters of the approximations are passed to the limit step by step in the order:
$$ (P_{\epsilon\delta\mu R}) \stackrel{\epsilon \to 0}{\longrightarrow} (P_{\delta\mu R})\stackrel{R \to \infty}{\longrightarrow}  (P_{\delta\mu})  \stackrel{\mu\to 0}{\longrightarrow} (P_{\delta})\stackrel{\epsilon\to 0}{\longrightarrow} (P).$$

 \begin{remark}\begin{enumerate}
 \item[(i)] Notice that the fractional operator is always defined by extending the function by $0$ outside the ball $B_R$ in the first two problems of the approximation $(P_{\epsilon\delta\mu R})$ and $(P_{\delta\mu R})$.
This is a delicate aspect which needs to be properly justified. The functions $U_1, U_2$ are defined on a ball $B_R$ and extended by $0$ to $\RN \setminus B_R$. We are able to do this extension since $U_1,U_2 \in H^1_0(B_R)$ by \eqref{eq:Lpaprox} therefore they have the right decay at the boundary $\partial B_R$ that allows the extension by $0$.  This is also one of the reasons for which the term $\delta \Delta U$ in the approximating problems  is the last one passing to the limit.
\item[(ii)] The term with $\delta $ coefficient in the $L^p$ estimate \eqref{eq:Lpaprox} gives $H^1$ regularity, an essential information in using parabolic compactness criteria. Again, this motivates the $\delta \to 0$ limit to be the last one.
\end{enumerate}
\end{remark}

\textbf{Passing to the limit. } First limit is done as $\epsilon \to 0$ and is based on the compactness criteria of type Simon-Aubin-Lions \cite{simon} in the context of
\begin{equation*}
H^{1-s}_{\epsilon_0}(B_R)\subset L^2(B_R) \subset H^{-1}(B_R),
\end{equation*}
where $H^{1-s}_{\epsilon_0}(B_R)$ is the space associated to \eqref{eq:aproxop}, and thus the left hand side inclusion is compact. We conclude that the family of approximate solutions $\{U_1\}_{\epsilon>0}$ is relatively compact in $L^2(0,T:L^2(B_R))$ and we obtain that
$(U_1)_{\epsilon,\delta, \mu, R} \to (U_2)_{\delta, \mu, R}$ as $\epsilon \to 0$ in $L^2(0,T:L^2(B_R)),$ up to subsequences.

As usual, the limit $U_2$ is identified to be a weak solution of a limit problem, in this case $(P_{\delta\mu R})$. Moreover, $U_2$ will satisfy the corresponding energy estimates which are proved by passing to the limit as $\epsilon \to 0$ the estimates for $U_1$. The following two limits $R \to \infty$ and $\mu \to 0$ are similar using the same type of compactness criteria of Simon.

The novelty appears in the last limit as $\delta \to 0$ where the regularity given by $H_1$ term with $\delta$ coefficient is lost. Here we need to use  a different compactness criteria due to Rakotoson and Temam \cite{RakotosonTemam} which does not ask for such strong regularity assumptions as before. We conclude that the solution $U_4$ of $(P_\delta)$ satisfies
\begin{equation*}
U_4 \rightarrow u  \text{ as } \delta \to 0 \text{ in }L^2_{\textup{loc}}(\RN \times (0,T)).
\end{equation*}

In the end we prove that $u$ is a weak solution to Problem \eqref{eq:maineq} and it satisfies the corresponding energy estimates. We call this $u$ \emph{constructed weak solution} since there is no uniqueness theory available in $\RN$. For $N=1$ we prove in Section \eqref{Section:Uniqueness} that uniqueness holds in the class of weak solutions and therefore the $u$ we have constructed is indeed the weak solution to Problem \eqref{eq:maineq}.

An $L^p$-$L^\infty$ smoothing effect is proved by combining $L^p$ energy estimate \eqref{eq:lpenergy} with the Nash-Gagliardo-Nirenberg inequality (See Theorem 7.4 in \cite{StTeVa17}) for the function $u^{\frac{m+p+1}{2}}$. More precisely we get
\begin{equation}\label{smoothform}
\| u(\cdot,t)\|_{L^{\infty}(\RN)} \le C_{N,s,m,p} \, t^{-\gamma_p}\|u_0\|_{L^p(\RN)}^{\delta_p} \quad \textup{for all} \quad t>0,
\end{equation}
where
$\gamma_p=\frac{N}{(m-1)N+2p(1-s)}$, $\delta_p=\frac{2p(1-s)}{(m-1)N+2p(1-s)}$.

\subsubsection{Existence for initial data in $\mathcal{M}_+(\R^N)$}

The existence of a solution for measure data is done via an approximating problem with data $(u_0)_n \in L^1(\RN) \cap L^\infty(\RN)$ where $(u_0)_n \to \mu$ and it conserves the mass $\|(u_0)_n\|_{L^1(\R^N)} =\mu(\R^N)$.
More precisely, let $u_n$ be the solution to Problem \eqref{eq:maineq} with data
$$
(u_0)_n(x):=\int_{\R^{N}}\rho_n(x-z)d \mu(z).
$$
We use the smoothing effect \eqref{smoothform} for $L^1(\R^N)\cap L^\infty(\R^N)$ initial data in the particular case $p=1$. Then, as in the previous section $u_n$ satisfies the energy estimates plus the smoothing effect:
\[
\|u_n(\cdot,t)\|_{L^\infty(\R^N)}\leq \|u_n(\cdot,\tau)\|_{L^\infty(\R^N)} \leq C_{N,s,m} \, \tau^{-\gamma}\|(u_0)_n\|_{L^1(\RN)}^{\delta}=C_{N,s,m} \, \tau^{-\gamma}\mu(\R^N)^{\delta},\]
where
$\gamma=\frac{N}{(m-1)N+2(1-s)}$, $\delta=\frac{2(1-s)}{(m-1)N+2(1-s)}$. Note that the bound does not depend on the approximation parameter $n$. In a similar way as before, we derive compactness estimates and apply the Rakotoson-Temam criteria \cite{RakotosonTemam} in order to obtain a limit as $n \to \infty$ away from $t=0$
 \begin{equation*}
u_n\longrightarrow u^\tau \quad \text{as}\quad  n \to \infty \quad  \text{in} \quad L^2_{\text{loc}}(\R^N\times(\tau,T)).
\end{equation*}
We also show that the initial data is recovered. Basically, the second energy estimate given by \eqref{eq:secondenergy}  allows us to prove that for any test function $\phi$ we have that
\begin{equation*}
\begin{split}
\left|\int_0^\tau\int_{\RN} u_n^{m-1} \nabla (-\Delta)^{-s} u_n  \nabla \phi dx dt \right|\leq \Lambda (\tau)
\end{split}
\end{equation*}
for some modulus of continuity that only depends on $\phi$, $\mu$ and $s$. Thus,
\begin{equation*}
\begin{split}
 \left|\int_{\R^N}(u_n(\tau)-(u_0)_n)\phi dx \right|&=  \left|\int_0^\tau\int_{\R^N}\partial_t u_n\phi \,dx dt\right|\\&=\left|\int_0^\tau\int_{\RN} u_n^{m-1} \nabla (-\Delta)^{-s} u_n  \nabla \phi \,dx dt \right|\leq \Lambda(\tau).
\end{split}
\end{equation*}
A standard diagonal argument in $\tau$ and $n$ completes the proof of existence for measure data.

$\bullet$ \textbf{Conservation of mass} is proved by using the previous estimate with the sequence of cutoff type test functions
 $\phi_R(x)=\phi(x/R)$ with $0\le \phi\le 1$ and $\phi_1(x)=1$ for $|x|\le 1$ and such that $\|\nabla \phi_R\|_{L^\infty(\R^N)}= O(R^{-1})$.

\subsection{Sketch of the proof of speed of propagation}

The proof requires delicate barrier arguments since Problem \eqref{eq:maineq} is proved to have a lack of comparison principle. We refer to \cite{CVf1} for an explicit example of this fact.

\subsubsection{Finite speed of propagation for $m\in[2,\infty)$}
However, a special kind of super solutions (so-called true super-solutions), are of particular interest. We can show that, comparing any solution with a true super-solution, no contact point between them is possible. This will be enough to show the property of finite speed of propagation.

Without loss of generality, we assume that $0\leq u_0\leq1$ (thus, $0\leq u(x,t)\leq 1$) and consider the parabola-like function
\[
U(x,t)=((Ct-(|x|-b))_+)^2.
\]
where $b>0$ is such that $u_0(x)< U(x,0)=:U_0(x)$ for all $x\in B_b(0)$ and $C$ is a suitable constant to be chosen later.

We argue by contradiction at a possible first contact point $(x_c,t_c)$ between $u$ and $U$.  The fact that such a first contact point happens for $t>0$ and $x\ne \infty$ is justified by regularization. We also exclude the extreme case where the contact is made at the boundary of the support of $U$ given by $|x_f(t_c)|:=b+Ct_c$ (see Lemma 7.2 in \cite{StTeVa16}). Then, there exists $h>0$ such that $b+Ct_c-|x_c|=h>0$.  At $(x_c,t_c)$, we have that $u=U$, $\nabla (u-U)=0$, $\Delta(u-U) \le 0$, $(u-U)_t \ge 0,$ that is
$$
u(x_c,t_c)=h^2,\quad u_r=-2h, \quad \Delta u \le 2N, \quad u_t\ge 2Ch,
$$
where $r=|x|$ denotes the radial coordinate. We also have the following estimates on $\nabla p:=\nabla (-\Delta)^{-s}$ for $0<s<1/2$ (see Theorem 5.1. of \cite{CVf1}):
\begin{equation*}
-{p_r}(|x_c|,t_c) \le K_1 + K_2 h^{1+2s} +K_3 h, \quad \Delta p(|x_c|,t_c) \le K_4
\end{equation*}
for some $K_1,K_2,K_3,K_4\geq0$. We now use the expanded form of Problem \eqref{eq:maineq} given by $u_t=(m-1)u^{m-2}\nabla u \cdot \nabla p + u^{m-1}\Delta p$, we get the inequality
\begin{equation*}
\begin{split}
C&\le (m-1)h^{2m-4}\left(-{p_r}(|x_c|,t_c) + \frac{h}{2}\Delta p(|x_c|,t_c)\right)\\
&\leq (m-1)h^{2m-4}\left( K_1 + K_2h^{1+2s} + (K_3+\frac{K_4}{2})h\right) ,
\end{split}
\end{equation*}
which leads to a contradiction choosing $C=C(s,N)$ large enough.

When $1/2\leq s< 1$, an improved version on the estimate of $p_r$ leads to similar result, but this time $C=C(t)$. This is again enough to prove the property of finite speed of propagation, but this time, we do not have a quantitative estimate on the growth of the support.

One can easily see that the term $h^{2m-4}$ in the last estimate needs $m\geq2$ to create a contradiction. In fact $m=2$ is show to be the critical exponent, as we show in the following section.

\subsubsection{Infinite speed of propagation}
In dimension $N=1$, we have already established a duality between weak solutions of \eqref{eq:maineq}
\[
\partial_t u = \nabla \cdot (u^{m-1} \nabla (-\Delta)^{-s}u)
\]
and viscosity solutions of the ``integrated problem"
\begin{equation}\label{eq:viscprob}
\partial_t v= -|v_x|^{m-1}\lapal v ,
\end{equation}
where $v(x,t)=\int_{-\infty}^x u(x,t)dx$ and $\alpha=1-s$. It will be enough to consider the initial data given by
\begin{equation}\label{u0Heaviside}v_0(x)\ge H_{x_0}(x)=\left\{
  \begin{array}{ll}
    0, & \hbox{$x< x_0$,} \\[2mm]
  1, & \hbox{$x>x_0$.}
  \end{array}\right.\end{equation}
Indeed, \eqref{eq:viscprob} has suitable comparison principles for viscosity solutions. See Proposition 8.5 and Proposition 8.6 in \cite{StTeVa16} for a standard comparison principle and a parabolic type comparison principle respectively.  At this point, we need to find a subsolution $\Phi=\Phi(x,t)$ of \eqref{eq:viscprob} such that $\Phi(x,0)\leq v_0(x)$ and $\Phi(x,t)>0$ for any $t>0$ and $|x|$ arbitrary large. See Figures \ref{FigBarrier0} and \ref{FigBarrier1} for a graphic version of the proof.

\definecolor{ffqqqq}{rgb}{1,0,0}
\definecolor{qqttcc}{rgb}{0,0.2,0.8}
\definecolor{sqsqsq}{rgb}{0.125,0.125,0.125}

\definecolor{ffqqqq}{rgb}{1,0,0}
\definecolor{qqqqcc}{rgb}{0,0,0.8}
\definecolor{sqsqsq}{rgb}{0.125,0.125,0.125}

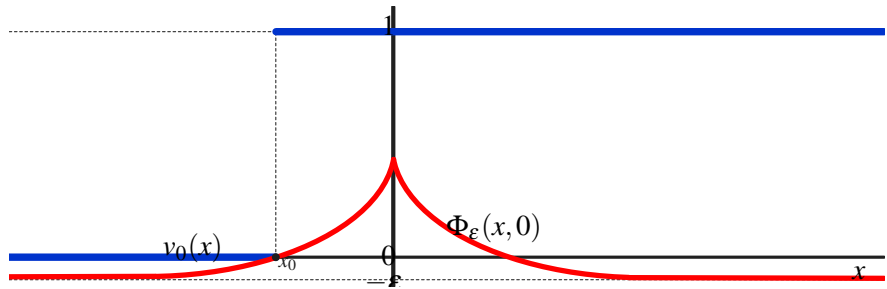
\begin{figure}[ht!]
\centering
\begin{tikzpicture}[line cap=round,line join=round,>=triangle 45,x=5.7cm,y=3cm]
\clip(-0.896,-0.128) rectangle (1.145,1.11);
\draw [line width=1.6pt,color=sqsqsq] (0,-0.128) -- (0,1.11);
\draw [line width=1.2pt,color=sqsqsq,domain=-0.896:1.145] plot(\x,{(-0-0*\x)/5.977});
\draw [dash pattern=on 1pt off 1pt,color=sqsqsq,domain=-0.896:1.145] plot(\x,{(--2-0*\x)/2});
\draw [line width=2.8pt,color=qqttcc,domain=-0.896473504307193:-0.27353490278664827] plot(\x,{(-0-0*\x)/-0.471});
\draw [dash pattern=on 1pt off 1pt,color=sqsqsq] (-0.274,0)-- (-0.274,1);
\draw [line width=2.8pt,color=qqttcc,domain=-0.27353490278664816:1.145157305834116] plot(\x,{(--0.844-0*\x)/0.844});
\draw [dash pattern=on 1pt off 1pt,color=sqsqsq,domain=-0.896:1.145] plot(\x,{(--0.092-0*\x)/-0.92});
\draw [shift={(-0.575,0.493)},line width=2pt,color=ffqqqq]  plot[domain=4.749:6.18,variable=\t]({1*0.579*cos(\t r)+0*0.579*sin(\t r)},{0*0.579*cos(\t r)+1*0.579*sin(\t r)});
\draw [line width=2pt,color=ffqqqq] (-0.554,-0.085)-- (-1.424,-0.09);
\draw [shift={(0.584,0.496)},line width=2pt,color=ffqqqq]  plot[domain=3.248:4.654,variable=\t]({1*0.587*cos(\t r)+0*0.587*sin(\t r)},{0*0.587*cos(\t r)+1*0.587*sin(\t r)});
\draw [line width=2pt,color=ffqqqq,domain=0.55:1.145157305834116] plot(\x,{(-0.087-0.005*\x)/0.989});
\draw (1.047,0.003) node[anchor=north west] {$\textcolor{black}{x}$};
\draw (0.1,0.24) node[anchor=north west] {$\textcolor{black}\Phi_\epsilon(x,0)$};
\draw (-0.561,0.15) node[anchor=north west] {$\textcolor{black}v_0(x)$};
\draw (-0.051,1.12) node[anchor=north west] {\textcolor{black}1};
\draw (-0.051,0.096) node[anchor=north west] {\textcolor{black}0};
\draw (-0.09,-0.025) node[anchor=north west] {$\textcolor{black}-\mathbb{\epsilon}$};
\begin{scriptsize}
\fill [color=sqsqsq] (-0.274,0) circle (1.5pt);
\draw[color=sqsqsq] (-0.245,-0.03) node {$x_0$};
\end{scriptsize}
\end{tikzpicture}
\caption{Comparison with the barrier at time $t=0$}
\label{FigBarrier0}
\end{figure}

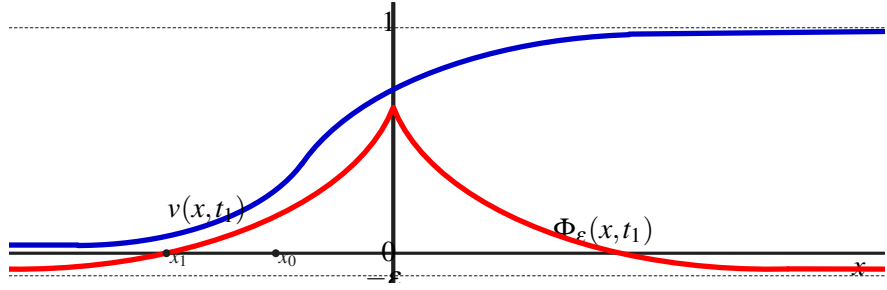
\begin{figure}
\centering
\begin{tikzpicture}[line cap=round,line join=round,>=triangle 45,x=5.7cm,y=3cm]
\clip(-0.896,-0.128) rectangle (1.145,1.11);
\draw [line width=1.6pt,color=sqsqsq] (0,-0.128) -- (0,1.11);
\draw [line width=1.2pt,color=sqsqsq,domain=-0.896:1.145] plot(\x,{(-0-0*\x)/5.977});
\draw [dash pattern=on 1pt off 1pt,color=sqsqsq,domain=-0.896:1.145] plot(\x,{(--2-0*\x)/2});
\draw [dash pattern=on 1pt off 1pt,color=sqsqsq,domain=-0.896:1.145] plot(\x,{(--0.092-0*\x)/-0.92});
\draw (1.047,0.003) node[anchor=north west] {$\textcolor{black}{x}$};
\draw (0.35,0.21) node[anchor=north west] {$\textcolor{black}\Phi_\epsilon(x,t_1)$};
\draw (-0.051,1.12) node[anchor=north west] {\textcolor{black}1};
\draw (-0.051,0.096) node[anchor=north west] {\textcolor{black}0};
\draw (-0.09,-0.025) node[anchor=north west] {$\textcolor{black}-\mathbb{\epsilon}$};
\draw [shift={(-0.74,0.598)},line width=2pt,color=qqqqcc]  plot[domain=4.721:5.917,variable=\t]({1*0.564*cos(\t r)+0*0.564*sin(\t r)},{0*0.564*cos(\t r)+1*0.564*sin(\t r)});
\draw [shift={(0.604,0.103)},line width=2pt,color=qqqqcc]  plot[domain=1.633:2.798,variable=\t]({1*0.868*cos(\t r)+0*0.868*sin(\t r)},{0*0.868*cos(\t r)+1*0.868*sin(\t r)});
\draw [line width=2pt,color=qqqqcc,domain=-0.896473504307193:-0.7353358406914835] plot(\x,{(-0.005-0*\x)/-0.138});
\draw [line width=2pt,color=qqqqcc,domain=0.55:1.145157305834116] plot(\x,{(--0.444--0.009*\x)/0.462});
\draw [shift={(-0.879,0.826)},line width=2pt,color=ffqqqq]  plot[domain=4.667:6.081,variable=\t]({1*0.897*cos(\t r)+0*0.897*sin(\t r)},{0*0.897*cos(\t r)+1*0.897*sin(\t r)});
\draw [shift={(0.879,0.826)},line width=2pt,color=ffqqqq]  plot[domain=3.343:4.759,variable=\t]({1*0.897*cos(\t r)+0*0.897*sin(\t r)},{0*0.897*cos(\t r)+1*0.897*sin(\t r)});
\draw [line width=2pt,color=ffqqqq,domain=0.92:1.145157305834116] plot(\x,{(-0.014-0.001*\x)/0.215});
\draw (-0.55,0.3) node[anchor=north west] {$\textcolor{black}v(x,t_1)$};
\begin{scriptsize}
\fill [color=sqsqsq] (-0.274,0) circle (1.5pt);
\draw[color=sqsqsq] (-0.245,-0.03) node {$x_0$};
\fill [color=sqsqsq] (-0.529,0) circle (1.5pt);
\draw[color=sqsqsq] (-0.499,-0.03) node {$x_1$};
\end{scriptsize}
\end{tikzpicture}\caption{Comparison with the barrier at time $t>0$} \label{FigBarrier1}
\end{figure}

We need to find this kind of subsolution. First, consider a function $G=G(x)$ such that $G$ is compactly supported in the interval $(-x_0,\infty)$, $G(x) \le C_1$ for all $x\in \mathbb{R}$ and $(-\Delta)^s G(x) \le -C_2|x|^{-(1+2s)}$ for all $x<x_0$ (see Lemma 9.1 in \cite{StTeVa16} for the existence of such a function $G$).
Now, given any $\tau,\xi,\epsilon>0$, we can find $C_2=C_2(N,s,\alpha,\tau)$ such that the function
\begin{equation*}
\Phi_\epsilon(x,t)=(t+\tau)^{b \gamma}\left((|x|+\xi)^{-\gamma}+G(x)\right)-\epsilon, \quad t\geq0, \ x\in \mathbb{R}.
\end{equation*}
satisfies
\begin{equation*}
(\Phi_\epsilon)_t  + |(\Phi_\epsilon)_x|^{m-1} \lapal \Phi_\epsilon \le 0 \quad \text{ for }x<x_0, \ t> 0
\end{equation*}
for $\displaystyle{\gamma=\frac{m+2\alpha}{2-m}}$ and $b=\frac{1}{m-1+2\alpha}$. The parameters $\gamma$ and $b$ are found in the study of self-similar solutions of \eqref{eq:viscprob}.

The main tool to finish the proof is given by the following parabolic comparison principle proved in \cite{StTeVa16}:
\begin{proposition}\label{ComparPrinc}
Let $m>1$, $\alpha \in (0,1)$ and $N=1$.  Let $v$ be a viscosity solution of Problem \eqref{IntegEq}-\eqref{initialv0}. Let  $\Phi: \mathbb{R}\times[0,\infty)\to \mathbb{R}$ such that $\Phi \in C^2(\Omega\times (0,T))$. Assume that
\begin{itemize}
  \item $\Phi_t +|\Phi_x|^{m-1}\lapal\Phi< 0 $ for $x\in \Omega$, $t\in [0,T]$;
  \item $\Phi(x,0) < v(x,0)$  for all $x\in \mathbb{R}$ (comparison at initial time);
  \item $\Phi(x,t) < v(x,t)$  for all $x \in \mathbb{R} \setminus \Omega$ and $t\in (0,T)$ (comparison on the parabolic boundary).
\end{itemize}
Then $\Phi(x,t) \le v(x,t)$ for all $x \in \mathbb{R}$, $t\in (0,T).$
\end{proposition}

At this point we need to show that $\Phi_\epsilon$ can be compared at initial time and also on the parabolic boundary.

 The initial data \eqref{u0Heaviside} naturally impose the following conditions on $\Phi_\epsilon$:
\begin{equation*}
 \xi > x_0+ \epsilon^{-\frac{1}{\gamma}}.
\end{equation*}
that ensures that $\Phi_\epsilon(x_0,0) <v_0(x_0)$. Now let $\displaystyle k_1 :=\min\{ v(x,t): \ x\ge x_0, \ 0< t\le T \}>0$ (we recall that $v\in C([0,T]:C(\R))$ and $v_0(x_0)=1$).
The condition
$
\Phi_\epsilon(x,t) < v(x,t)$ \text{ for all } $x\ge x_0, \ t\in [0,T]$
requires
$$
 (t+1)^{b \gamma} (\xi^{-\gamma}+C_1 ) < k_1.
$$
The maximum value of $t=T$ for which this inequality holds is
\begin{equation*}
T<  \left(\frac{k_1}{ \xi^{-\gamma}+C_1 } \right)^{1/b\gamma}-1.
\end{equation*}
Thus, in order to have $T>0$ we require
$
\xi > (k_1-C_1)^{-\frac{1}{\gamma}}.
$
The remaining parameter $C_1$ from assumption (G2) is chosen here such that: $C_1<k_1$. By Proposition \ref{ComparPrinc} we obtain the desired comparison:
$$
v(x,t) \ge \Phi_\epsilon(x,t) \quad \text{for all } (x,t) \in Q_T.
$$
Now, let $x_1 < x_0<0$ and $t_1 \in (0,T)$ be arbitrary. It is now straightforward to show that for
\begin{equation*}
\epsilon  < \left[\frac{(t_1+1)^b}{(k_1-C_1)^{-\frac{1}{\gamma}} -x_1}\right]^\gamma.
\end{equation*}
we have that $\Phi_\epsilon(x_1,t_1)>0$ and thus, by comparison $v(x_1,t_1)>0$. In this way, we have proves the following result:
\begin{theorem}[\textbf{Infinite speed of propagation for $v$}]
Assume that $u_0\in L^\infty(\R^N)$ is nonnegative and compactly supported. Let $v$ be the solution of Problem \eqref{IntegEq}-\eqref{initialv0}. Then $0<v(x,t)<M$ for all $t>0$ and $x\in \R$.
\end{theorem}

The result for $u$ follows immediately. We have proved that $v(x,t)$ is positive for every $t>0$ and $x \in \R$, thus $u$ has accumulated mass at every $(x,t)$. This fact ensures that for every time $t>0$ there exists an $x\in \R$  arbitrary far from the origin such that $u(x,t)>0$. Moreover, when $u_0$ is radially symmetric and non-increasing in $|x|$ then $u$ inherits these properties, ensuring that $u$ can not take zero values.

\begin{remark} \begin{enumerate}
\item[(i)] This method is working only in one dimension since we use the integrated function. It is an  open issue the proof of infinite speed of propagation in dimension $N\ge 2$. New methods should be employed and, at least for particular cases of data, one can see a possible direction to continue: for instance radial data will produce radial solutions and then one could search for a suitable transformation between \eqref{eq:maineq} and a $1-D$ problem.
\item[(ii)]Infinite speed of propagation holds for $m<2$ for any self-similar solution as a consequence of the transformation formula from the previous section. Moreover their properties are imported from the alternative model \eqref{FPME}.
\end{enumerate}
\end{remark}

\subsection{Proof of the asymptotic behavior}

Here we provide the proof of   the asymptotic behavior in dimension $N=1$ using a four step method. This will be a new contribution to the study of Problem \eqref{eq:maineq}. The result can only be presented in dimension 1 due to the lack of uniqueness for \eqref{eq:maineq}. However, we will present intermediate steps valid in $\R^N$, and the reader could trivially see that the asymptotic behaviour result for general $N\geq1$ would follow from a result of uniqueness of solutions with Dirac delta type initial data.
\subsubsection{Existence of a rescaled solution}
\begin{lemma}\label{lem:ula}
Let $m\in(1,+\infty)$, $s \in (0,1)$ and $N\geq1$. Assume that $u_0\in L^1(\R^N)$ and let $u$ be the constructed weak solution of \eqref{eq:maineq} given by Theorem \ref{ThmExistL1}. Then, for any $\lambda>0$  the rescaled function
\begin{equation*}
\ula(x,t)=\lambda^Nu(\lambda x, \lambda^bt),
\end{equation*} with $b = (m-1)N + 2-2s$, is a weak solution of
\begin{equation}\label{eq:ulambda}
   \left\{ \begin{array}{ll}
  \partial_t \ula = \nabla \cdot (\ula^{m-1} \nabla (-\Delta)^{-s}\ula)    &\text{for } x \in \RN, \, t>0,\\[2mm]
  \ula(0,x)  =\lambda^Nu_0(\lambda x) &\text{for } x \in \RN.
    \end{array}  \right.
\end{equation}
Moreover, $\ula$ has the following properties:
\begin{enumerate}
\item \textbf{(Conservation of mass)} For all $0 < t < T$ we have
$\displaystyle{
\int_{\RN}\ula(x,t)dx= \int_{\RN}u_0(x)dx.
}$

\item \textbf{($L^p$ energy estimate)} For all $1<p<\infty$ and $0 <\tau< t < T$ we have
\begin{equation*}
\begin{split}\intr \ula^p(x,t)dx + \frac{4p(p-1)}{(m+p-1)^2}\int_\tau^t &\int_{\RN}\Big|(-\Delta)^{\frac{1-s}{2}} \left[\ula^{\frac{m+p-1}{2}}\right](x,s)\Big|^2dxds \\
&\le \intr \ula^p(x,\tau)dx .
\end{split}
\end{equation*}

\item \textbf{(Second energy estimate)} For all $0 < \tau<t < T$  we have
\begin{equation*}
\begin{split}
  \frac{1}{2} \intr \left|(-\Delta)^{-\frac{s}{2}} \ula(x,t)\right|^2 dx  +
&\int_\tau^t \intr  u^{m-1} \left| \nabla  (-\Delta)^{-s} \ula(x,s)\right|^2 dx  ds\\
& \le \frac{1}{2} \intr \left|(-\Delta)^{-\frac{s}{2}} \ula(x,\tau)\right|^2 dx.
\end{split}
\end{equation*}

\item \textbf{(Smoothing effect)} For all $t>0$,
\[
\| \ula(\cdot,t)\|_{L^{\infty}(\RN)} \le C_{N,s,m} \, t^{-\gamma} \|u_0\|_{L^1(\R^N)}^\delta
\]
 where $\gamma=\frac{N}{(m-1)N+2(1-s)}>0$ and $\delta=\frac{2(1-s)}{(m-1)N+2(1-s)}>0$.

\end{enumerate}

\end{lemma}

Note that estimates 2 and 3 in Lemma \ref{lem:ula} are not uniform in $\lambda$ up to $\tau=0$ since the hypothesis $u_0\in L^1(\R^N)$ is not enough to find a uniform bound for the right hand side term. Note that $\ula$ only belongs to $L^1(\R^N)$ and $\ula(x,0)$ will converge to $\|u_0\|_{L^1(\R^N)}\delta_0$ as $\lambda \to \infty$. However, for any $\tau>0$, the smoothing effect ensures that $u_\lambda$ is bounded uniformly in $\lambda$ and then, the right hand side terms of estimates 2 and 3 can be bounded by the terms involving only $L^1$ norm of $u_0$. This kind of uniform estimates are very useful and will be given in more details later.

\begin{proof}  \textbf{I. $\ula$ is a weak solution of \eqref{eq:ulambda}. }Note that given any test function $\phi \in C_c^\infty(\R^N \times (0,T))$, we can define $\psi \in C_c^\infty(\R^N\times(0,\lambda^b T))$ such that $\phi(x,t)=\psi( \lambda x, \lambda^b t )$. Then, the first term in the weak formulation reads
\begin{equation*}
\begin{split}
\int_0^T\int_{\RN} \ula(x,t) \phi_t(x,t)\,dxdt&=\lambda^N\int_0^T\int_{\RN} u(\lambda x,\lambda^bt) \phi_t(x,t)\,dxdt\\
&=\lambda^{N+b}\int_0^T\int_{\RN} u(\lambda x,\lambda^bt) \psi_t(\lambda x,\lambda^bt)\,dxdt\\
&=\int_0^{\lambda^bT}\int_{\RN} u(y,s) \psi_s(y,s)\,dyds.
\end{split}
\end{equation*}
The second term is as follows:
\begin{equation*}
\begin{split}
\int_0^T\int_{\RN}  &\ula^{m-1}(x,t) \nabla (-\Delta)^{-s} \ula(x,t) \cdot\nabla \phi(x,t) \,dxdt\\
&=\lambda^{N(m-1)+N}\int_0^T\int_{\RN}  u^{m-1}(\lambda x,\lambda^b t) \nabla (-\Delta)^{-s}[u(\lambda\cdot, \lambda^b t )] \cdot\nabla \phi(x,t) \,dxdt\\
&=\lambda^{Nm+1-2s}\int_0^T\int_{\RN}  u^{m-1}(\lambda x,\lambda^b t) \nabla (-\Delta)^{-s}u(\lambda x, \lambda^b t ) \cdot\nabla \phi(x,t) \,dxdt\\
&=\lambda^{Nm+2-2s}\int_0^T\int_{\RN}  u^{m-1}(\lambda x,\lambda^b t) \nabla (-\Delta)^{-s}u(\lambda x, \lambda^b t ) \cdot\nabla \psi(\lambda x,\lambda^bt) \,dxdt\\
&= \lambda^{N(m-1)+2-2s-b}\int_0^{\lambda^bT}\int_{\RN}  u^{m-1}(y,s) \nabla (-\Delta)^{-s}u(y, s ) \cdot\nabla \psi(y,s) \,dyds\\
&=\int_0^{\lambda^bT}\int_{\RN}  u^{m-1}(y,s) \nabla (-\Delta)^{-s}u(y, s ) \cdot\nabla \psi(y,s) \,dyds.
\end{split}
\end{equation*}
Finally  the initial condition is reformulated as
\begin{equation*}
\int_{\RN} \ula(x,0) \phi(x,0)dx=\lambda^N\int_{\RN} u_0(\lambda x) \phi(x,0)dx=\int_{\RN} u_0(y) \psi(y,0)dy,
\end{equation*}
which concludes the proof of (I).

\textbf{II. $\ula$ has conservation of mass independent of $\lambda$. } Since $u$ preserves the mass, we have that
\begin{equation*}
\int_{\R^N}\ula(x,t)dx=\lambda^N\int_{\R^N}u(\lambda x,\lambda^bt)dx=\int_{\R^N}u(y,\lambda^bt)dx=\int_{\R^N}u_0(x)dx.
\end{equation*}

\textbf{III. Energy estimates.} The energy estimates are obtained by similar scaling arguments using the energy estimates available for $u$.

\textbf{IV. $\ula$ has smoothing effect uniform in $\lambda$. } Since $\ula$ is a weak solution of \eqref{eq:maineq} we can use the smoothing effect of Theorem \ref{ThmExistL1} together with the result of conservation of mass independent of $\lambda$ to get
\begin{equation*}
\| \ula(\cdot,t)\|_{L^{\infty}(\RN)} \le C_{N,s,m} \, t^{-\gamma}\|\ula(x,0)\|_{L^1(\RN)}^{\delta}=C_{N,s,m} \, t^{-\gamma}\|u_0\|_{L^1(\RN)}^{\delta} \quad \textup{for all } \, t>0.
\end{equation*}

\end{proof}

\subsubsection{Convergence of the rescaled solution as $\lambda\to\infty$}
\begin{lemma}\label{lem:ulalimit}
Let $m\in(1,+\infty)$,  $s \in (0,1)$ and $N\geq1$. Assume that $u_0\in L^1(\R^N)$ such that $\|u_0\|_{L^1(\R^N)}=M$. Let also $\ula$ defined as in Lemma \ref{lem:ula}. Then,  for any $0<t_1<t_2<\infty$, there exists a function $U_M\in L^1(\R^N\times(0,\infty))$ and a subsequence $\{\lambda_j\}_{j=1}^\infty$ such that
\begin{equation}\label{Lpconv}
u_{\lambda_j} \to U_M  \quad \textup{as} \quad \lambda_j \to \infty \quad \textup{in} \quad L^p(\R^N\times [t_1,t_2]) \quad \textup{for} \quad 1\leq p <\infty,
\end{equation}
where $U_M$ is a weak solution of \eqref{eq:maineq} with measure initial data $U_M(x,0)=M \delta_0$ and it satisfies the properties 1-5 of Theorem \ref{ThmExistL1}. Moreover, in dimension $N=1$, the full sequence $\ula$ converge in the sense of \eqref{Lpconv}.
\end{lemma}
\begin{proof}
\textbf{I. Existence of a limit. }Estimates 1-3 from Lemma \ref{lem:ula} are enough to follow the same proof of Theorem 5.2 in \cite{StTeVa17}. We get that (up to a subsequence),
\[
u_{\lambda_j} \to U_M  \quad \textup{as} \quad \lambda_j \to \infty \quad \textup{in} \quad L^2_{\textup{loc}}(\R^N\times [t_1,t_2])
\]
where $U_M$ is a weak solution of \eqref{eq:maineq} with initial data $U_M(x,0)=M \delta_0$ with all the desired properties. Now we need to have a uniform control of the tails of the solutions in order to be able to pass from local convergence in $L^2$ to global convergence in any $L^p$. Moreover, in dimension $N=1$, uniqueness of weak solutions ensures that the full sequence $\ula$ converges.

\textbf{II. Tail control. }
Let $\phi \in C^\infty(\R^N)$ be a nondecreasing function such that $\phi(x)=0$ if $|x|<1$ and $\phi(x)=1$ if $|x|>2$. Now we take $\phi_R(x):=\phi(x/R)$ as test function (after an approximation argument) to get
\begin{equation*}
\begin{split}
&\int_{\R^N} \ula (x,t) \phi_R(x)dx\\
&\leq \int_{\R^N} \ula (x,0) \phi_R(x)dx - \int_{\R^N}\int_0^t \ula^{m-1}(x,t) \nabla (-\Delta)^{-s} \ula(x,t) \cdot \nabla \phi_R(x)dx=I+II.
\end{split}
\end{equation*}
First, we note that since $u_\lambda \geq0$, then
\begin{equation*}
\int_{\R^N} \ula (x,t) \phi_R(x)dx=\int_{|x|>R} \ula (x,t) \phi_R(x)dx\geq\int_{|x|>2R} \ula (x,t) dx.
\end{equation*}
On the other hand, for $\lambda>1$ we have
\begin{equation*}
\int_{\R^N} \ula (x,0) \phi_R(x)dx\leq \lambda^N\int_{|x|>R} u_0 (\lambda x) dx=\int_{|y|>\lambda R} u_0 (y) dy\leq \int_{|y|>R} u_0 (y) dy
\end{equation*}
and the last term clearly goes to zero as $R\to \infty$ since $u_0\in L^1(\R^N)$. We also have, as in part III in the proof of Theorem 5.2 in \cite{StTeVa17}, that
\begin{equation*}
\begin{split}
|II| \leq\|\nabla \phi_R\|_{\infty}\int_0^t\int_{\RN} \ula^{m-1}(x,t) |\nabla (-\Delta)^{-s} \ula(x,t)|dx dt\leq \Lambda (t) /R
\end{split}
\end{equation*}
where $\Lambda$ is a locally bounded function in $t$. Combining the above estimates, we conclude that
$\int_{|x|>2R} \ula (x,\tau) dx \to 0$ as $R\to \infty$ for all $\tau\in(0,t)$. Passing to the limit, the same estimate is inherited by $U_M$. A similar tail control argument has been used by one of the authors in \cite{IgnatStan} for a fractional diffusion-convection equation.

\textbf{III. Convergence in $L^p(\R^N \times[t_1,t_2])$. } First, we prove  $L^1$ convergence. From step II in this proof, for any $\epsilon>0$ we can choose $R$ large enough such that
\[
\begin{split}
I_{B_R^c}&=\int_{t_1}^{t_2}\int_{B_R^c}|u_{\lambda_j}(x,t)-U_M(x,t)|dxdt\\
&=\int_{t_1}^{t_2}\int_{B_R^c}|u_{\lambda_j}(x,t)|dxdt+\int_{t_1}^{t_2}\int_{B_R^c}| U_M(x,t)|dxdt<\epsilon/2.
\end{split}
\]
On the other hand,
\begin{equation*}
\begin{split}
I_{B_R}&=\int_{t_1}^{t_2}\int_{B_R}|u_{\lambda_j}(x,t)-U_M(x,t)|dxdt\\
&\leq |B_R|(t_2-t_1)\left(\int_{t_1}^{t_2}\int_{B_R}|u_{\lambda_j}(x,t)-U_M(x,t)|^2dxdt\right)^{\frac{1}{2}}.
\end{split}
\end{equation*}
Since $u_{\lambda_j} \to U_M$ \textup{as}  $\lambda_j \to \infty$  \textup{in} $L^2_{\textup{loc}}(\R^N\times [t_1,t_2])$, we can now choose $\lambda_j$ big enough such that $I_{B_R}<\epsilon/2$. In this way,
\begin{equation*}
\|u_{\lambda_j} - U_M\|_ {L^1(\R^N\times [t_1,t_2])}\leq I_{B_R}+I_{B_R^c}<\epsilon
\end{equation*}
which concludes the proof of $L^1$ convergence. By the smoothing effect, both $\ula$ and $U_M$ are uniformly bounded outside $t=0$, for any $R>0$ we have that
\begin{equation*}
\begin{split}
\|u_{\lambda_j} - U_M\|^p_ {L^p(\R^N\times [t_1,t_2])}&\leq \|u_{\lambda_j} - U_M\|^{p-1}_ {L^\infty(\R^N\times [t_1,t_2])}\|u_{\lambda_j} - U_M\|_ {L^1(\R^N\times [t_1,t_2])}\\
&\leq C\|u_{\lambda_j} - U_M\|_ {L^1(\R^N\times [t_1,t_2])}
\end{split}
\end{equation*}
which again converges to 0 as $\lambda_j\to \infty$.
\end{proof}

\subsubsection{Self-similarity of the limit solution}
\begin{lemma}[Existence of self-similar solution]\label{LemmaSelfSim}
Let $m\in(1,+\infty)$, $s \in (0,1)$, $N=1$. The solution $U_M$ constructed in Lemma \ref{lem:ulalimit} is a selfsimilar solution of the form
\begin{equation*}
U_M(x,t)=t^{-\alpha}\phi(xt^{-\beta}),
\end{equation*}
for a certain function $\phi\in \R^N \to \R$ and $\alpha=N\beta$ with $\beta=1/(N(m-1)+2-2s)$.
\end{lemma}
\begin{proof} First we note that for $b=N(m-1)+2-2s$, $U_M$ is invariant under the following scaling
\begin{equation}\label{eq:rescaled}
\begin{split}
\lambda_0^NU_M(\lambda_0x,\lambda_0^bt)&=\lim_{\lambda\to\infty}\lambda_0^N\ula(\lambda_0x,\lambda_0^bt)=\lim_{\lambda\to\infty}(\lambda_0\lambda)^Nu(\lambda_0\lambda x,(\lambda_0\lambda)^bt)\\
&=\lim_{\lambda\lambda_0\to\infty}u_{\lambda_0\lambda}( x,t)=U_M(x,t).
\end{split}
\end{equation}
Thus, since $b\beta=1$ and then, we choose $\lambda_0=t^{-\frac{1}{b}}=t^{-\beta}$ to get
\begin{equation*}
\begin{split}
U_M(x,t)&=\lambda_0^NU_M(\lambda_0 x,\lambda_0^bt)=t^{-\frac{N}{b}}U_M( x t^{-\frac{1}{b}},1)\\
&=t^{-\frac{N}{b}}U_M( x t^{-\frac{1}{b}},1)=t^{-\alpha}U_M( x t^{-\beta},1).
\end{split}
\end{equation*}
\end{proof}

\begin{remark}
Note that, in identity \eqref{eq:rescaled}, the fact the full sequence $\ula$ converges plays a crucial role. If we do not have this property, we cannot ensure that the sequence in $\{\lambda_0\lambda_j\}_{j=1}^\infty$ gives a convergent $u_{\lambda_0 \lambda_j}$.
\end{remark}

\subsubsection{Proof of Theorem \ref{thm:main}}
We have that
\begin{equation*}
\|\ula-U_M\|_{L^p(\R^N \times [t_1,t_2])}\to 0 \quad \textup{as} \quad \lambda\to \infty,
\end{equation*}
which in particular implies
\begin{equation*}
\|\ula(\cdot,t)-U_M(\cdot,t)\|_{L^p(\R^N)}\to 0 \quad \textup{as} \quad \lambda\to \infty \quad \textup{for a.e} \quad t\in[t_1,t_2].
\end{equation*}
Without loos of generality, assume the above limit holds for $t=1$. Then, choosing $\tau=\lambda^b=\lambda^{1/\beta}$,
\begin{equation*}
\begin{split}
\|\ula(\cdot,1)-U_M(\cdot,1)\|_{L^p(\R^N)}&=\lambda^N\| u(\lambda \cdot, \lambda^b)-U_M(\lambda \cdot, \lambda^b )\|_{L^p(\R^N)}\\
&=\lambda^{N-\frac{N}{p}}\| u( \cdot, \lambda^b)-U_M(\cdot, \lambda^b )\|_{L^p(\R^N)}\\
&=\tau^{N(1-\frac{1}{p})\beta}\|u( \cdot, \tau)-U_M(\cdot, \tau )\|_{L^p(\R^N)}.
\end{split}
\end{equation*}
Since $\tau \to \infty$ as $\lambda\to\infty$, we conclude that
\begin{equation*}
\tau^{\frac{N(1-\frac{1}{p})}{(m-1)N+2-2s}}\|u( \cdot, \tau)-U_M(\cdot, \tau )\|_{L^p(\R^N)}\to0 \quad \textup{as} \quad \tau \to \infty.
\end{equation*}

\section{Comments and open problems}

As a summary,  

 \noindent  $\bullet$ We establish the theory of existence of suitable weak solutions of problem (M1) and settle the question of finite vs infinite speed of propagation depending on the parameter $m$. We also settle the asymptotic behavior in one dimension by means of an integrated version of the problem.

 \noindent $\bullet$ The questions of uniqueness in several dimensions are widely open and ought to be addressed. Once this result is available, the existence of selfsimilar solutions together with the asymptotic behaviour would follow with the techniques showed in this paper.

If $m\in(1,2), \ N\ge 1$ we have uniqueness by the 1-to-1 correspondence of self-similar solutions between \eqref{eq:maineq} and \eqref{FPME}, and the last ones are known to be unique. However, we need a uniqueness result regarding general initial data in order to be able to prove the asymptotic behaviour in dimension higher than one.

If $m\in(1,\infty)$ and $N=1$, the solutions (not only self-similar ones) are unique since there exists a 1-to-1 correspondence with viscosity solutions of the integrated problem, which are known to be unique.

 \noindent $\bullet$ Another pending issue is continuity of weak solutions. In the case $m=2$ H\"older continuity is proved in  \cite{CSV, CVf2}.

 \noindent $\bullet$  Recently, the problem posed in a bounded domain was considered in \cite{NguyenVaz} for dimension $N\ge 1$. Further work is to be done on that issue.

\noindent $\bullet$ Satisfying numerical experiments have been performed, see \cite{StTeVa17} for some numerical experiments using ideas of \cite{Te14}. A systematic and rigorous numerical analysis is still pending.

\medskip

{\sc Acknowledgments. } J.L.V. is partially supported by Spanish Project MTM2014-52240-P.
D. Stan was partially supported by the MEC-Juan de la Cierva postdoctoral fellowship number FJCI-2015-25797 and by the projects ERCEA Advanced Grant
2014 669689 - HADE, by the MINECO project MTM2014-53850-P, by Basque Government project IT-641-13 and also by the Basque Government through the BERC 2014-2017 program and by Spanish Ministry of Economy and Competitiveness MINECO: BCAM Severo Ochoa excellence accreditation SEV-2013-0323.
 F.d.T. is partially supported by the Toppforsk (research excellence) project Waves and Nonlinear Phenomena (WaNP), grant no. 250070 from the Research Council of Norway and by the ERCIM ``Alain Bensoussan'' Fellowship programme.







\end{document}